\documentclass[12pt]{amsart}
\usepackage{amssymb,latexsym,amsmath,amscd,amsthm,graphicx, color}
\usepackage[all]{xy}
\usepackage{hyperref}
\usepackage{mathrsfs}
\definecolor{uuuuuu}{rgb}{0.26666666666666666,0.26666666666666666,0.26666666666666666}
\definecolor{xdxdff}{rgb}{0.49019607843137253,0.49019607843137253,1.}
\definecolor{ffqqqq}{rgb}{1.,0.,0.}

\definecolor{uuuuuu}{rgb}{0.26666666666666666,0.26666666666666666,0.26666666666666666}
\definecolor{qqwuqq}{rgb}{0.,0.39215686274509803,0.}
\definecolor{zzttqq}{rgb}{0.6,0.2,0.}
\definecolor{xdxdff}{rgb}{0.49019607843137253,0.49019607843137253,1.}
\definecolor{qqqqff}{rgb}{0.,0.,1.}
\definecolor{cqcqcq}{rgb}{0.7529411764705882,0.7529411764705882,0.7529411764705882}
\definecolor{sqsqsq}{rgb}{0.12549019607843137,0.12549019607843137,0.12549019607843137}

\theoremstyle{plain}

\newtheorem{theorem}[subsection]{Theorem}

\newtheorem{lemma}[subsection]{Lemma}

\newtheorem{prop}[subsection]{Proposition}

\theoremstyle{definition}
\newtheorem{remark}[subsection]{Remark}

\newcommand{\uu}{\cup}
\newcommand{\ii}{\cap}
\newcommand{\UU}{\bigcup}
\newcommand{\II}{\bigcap}

\newcommand{\sci}{\subset}
\newcommand{\es}{\emptyset}
\newcommand{\set}[1]{\{#1\}}

\newcommand{\ga}{\alpha}

\newcommand{\gd}{\delta}
\newcommand{\gh}{\eta}

\newcommand{\gk}{\kappa}

\newcommand{\gn}{\nu}
\newcommand{\go}{\omega}

\newcommand{\gs}{\sigma}
\newcommand{\gt}{\tau}

\newcommand{\gG}{\Gamma}

\newcommand{\gO}{\Omega}

\newcommand{\tbf}{\textbf}
\newcommand{\tit}{\textit}

\newcommand{\C}[1]{\mathcal{#1}}
\newcommand{\D}[1]{\mathbb{#1}}
\newcommand{\F}[1]{\mathfrak{#1}}

\newcommand{\te}{\text}

\newcommand{\ep}{\epsilon}

\newcommand{\ol}{\overline}
\newcommand{\ul}{\underline}

\newcommand{\nd}{\noindent}

\begin{document}

\noindent To appear, The Journal of Geometric Analysis

\title{Local dimensions and quantization dimensions in dynamical systems}


 \author{Mrinal Kanti Roychowdhury$^1$}
\address{$^1$School of Mathematical and Statistical Sciences\\
University of Texas Rio Grande Valley\\
1201 West University Drive\\
Edinburg, TX 78539-2999, USA.}
\email{mrinal.roychowdhury@utrgv.edu}

 \author{Bilel Selmi$^2$}
\address{$^2$
Analysis, Probability $\&$ Fractals Laboratory: LR18ES17\\
Department of Mathematics\\
Faculty of Sciences of Monastir\\
 University of Monastir \\
5000-Monastir, Tunisia.}
\email{bilel.selmi@fsm.rnu.tn}


\thanks{}
\date{}
\maketitle

\pagestyle{myheadings} \markboth{Mrinal Kanti Roychowdhury and Bilel Selmi}{Local dimensions and quantization dimensions in dynamical systems}

\begin{abstract}
Let $\mu$ be a Borel probability measure generated by a hyperbolic
recurrent iterated function system defined on a nonempty compact
subset of $\mathbb R^k$. We study the Hausdorff and the
packing dimensions, and the quantization dimensions of $\mu$ with
respect to the geometric mean error. The results
establish the connections with various dimensions of the measure
$\mu$, and generalize many known results about local dimensions and
quantization dimensions of measures.
\end{abstract}

\section{Introduction}
 Given a Borel probability measure $\mu$ on $\D R^k$, where $k\in \D N$, the $n$th quantization error for $\mu$ with respect to the geometric mean error is given by
\begin{equation} \label{eq00} e_n(\mu):=\inf \Big\{\exp \int \log d(x, \ga) d\mu(x) : \ga \sci \D R^k, \, 1\leq\te{card}(\ga) \leq n\Big\},\end{equation}
where $d(x, \ga)$ denotes the distance between $x$ and the set $\ga$ with respect to an arbitrary norm $d$ on $\D R^k$. A set $\ga$ for which the infimum is achieved and contains no more than $n$ points is called an \textit{optimal set of $n$-means} for $\mu$, and the collection of all optimal sets of $n$-means for $\mu$ is denoted by $\C C_n(\mu)$. Under some suitable restriction $e_n(\mu)$ tends to zero as $n$ tends to infinity. Following \cite{GL2}, we write
\begin{equation*} \hat e_n:=\hat e_n(\mu)=\log e_n(\mu)=\inf \Big\{\int \log d(x, \ga) d\mu(x) : \ga \sci \D R^k, \, 1\leq\te{card}(\ga) \leq n\Big\}.\end{equation*}
The numbers
\[\ul D(\mu):=\liminf_{n \to \infty}  \frac{\log n}{-\hat e_n(\mu)} \te{ and } \ol D(\mu):=\limsup_{n\to \infty} \frac{\log n}{-\hat e_n(\mu)},\]
are called the \tit{lower} and the \tit{upper quantization dimensions} of $\mu$ (of order zero), respectively. If $\ul D (\mu)=\ol D (\mu)$, the common value is called the \tit{quantization dimension} of $\mu$ and is denoted by $D(\mu)$. The quantization dimension measures the speed at which the specified measure of the error tends to zero as $n$ tends to infinity. This problem arises in signal processing, data compression, cluster analysis, and pattern recognition, and it also has been studied in the context of economics, statistics, and numerical integration (see \cite{BW, GG, GN, P, Z}). The quantization dimension with respect to the geometric mean error can be regarded as a limit state of that based on $L_r$-metrics as $r$ tends to zero (see \cite[Lemma 3.5]{GL2}). The following proposition gives a characterization of the lower and the upper quantization dimensions.
\begin{prop} (see \cite[Proposition 4.3]{GL2}) \label{prop1}
Let $\ul D=\ul D(\mu)$ and $\ol D=\ol D(\mu)$.

$(a)$ If $0\leq t<\ul D<s$, then
\[\lim_{n\to \infty} (\log n +t\hat e_n(\mu))=+\infty \te{ and } \liminf_{n\to \infty} (\log n+s\hat e_n(\mu)) =-\infty.\]

$(b)$ If $0\leq t<\ol D<s$, then
\[\limsup_{n\to \infty} (\log n +t\hat e_n(\mu))=+\infty \te{ and } \lim_{n\to \infty} (\log n+s\hat e_n(\mu)) =-\infty.\]
\end{prop}
For any $\gk>0$, the two numbers $\liminf\limits_{n\to \infty}  n^{1/\gk} e_n(\mu)$ and  $\limsup\limits_{n\to \infty}  n^{1/\gk} e_n(\mu)$ are, respectively, called the \tit{$\gk$-dimensional lower} and  the \tit{upper quantization coefficients} for $\mu$ with respect to the geometric mean error. For every $x\in \D R^k$, the \tit{lower} and the \tit{upper local dimensions} of the measure $\mu$ at $x$ are defined, respectively, by
\[\ul d_\mu(x)=\liminf_{r\to 0} \frac{\log \mu(B(x, r))}{\log r}\te{ and } \ol d_\mu(x)=\limsup_{r\to 0} \frac{\log \mu(B(x, r))}{\log r},\]
where $B(x, r)$ is the ball of radius $r$ centered at $x$. We say that the local dimension exists at $x$ if $\ul d_\mu(x)$ and $\ol d_\mu(x)$ are equal, and write $d_\mu(x)$ for the common value. These local dimensions, also known as pointwise dimensions, describe the power law behavior of $\mu(B(x, r))$ for small $r$, with $d_\mu(x)$ small if $\mu$ is `highly concentrated' near $x$. Notice that $d_\mu(x)=\infty$ if $x$ is outside the support of $\mu$ and $d_\mu(x)=0$ if $x$ is an atom of $\mu$. The \tit{lower} and the \tit{upper Hausdorff dimensions} of $\mu$ are defined, respectively, by
\begin{align*} \te{dim}_\ast \mu & =\inf \{\te{dim}_{\te{H}} E : E \te{ is Borel with } \mu(E)>0\}, \te{ and } \\
 \te{dim}^\ast \mu& =\inf \{\te{dim}_{\te{H}} E : E \te{ is Borel with } \mu(E)=1\}.
\end{align*}
Analogously, we define the \tit{lower} and the \tit{upper packing dimensions} of $\mu$, respectively, by
\begin{align*} \te{Dim}_\ast \mu & =\inf \{\te{dim}_{\te{P}} E : E \te{ is Borel with } \mu(E)>0\}, \te{ and } \\
 \te{Dim}^\ast \mu& =\inf \set{\te{dim}_{\te{P}} E : E \te{ is Borel with } \mu(E)=1}.
\end{align*}
Clearly, $\te{dim}_\ast \mu \leq \te{dim}^\ast \mu$ and $\te{Dim}_\ast \mu \leq \te{Dim}^\ast \mu$. When the equalities $\te{dim}_\ast \mu = \te{dim}^\ast \mu$ and $\te{Dim}_\ast \mu =\te{Dim}^\ast \mu$ are satisfied, we denote by $\te{dim}_{\te{H}} \mu$ and $\te{dim}_{\te{P}} \mu$, respectively, the Hausdorff and the packing dimensions of the measure $\mu$. Hausdorff dimension and packing dimension of a measure are closely related to lower local dimension and upper local dimension of the measure. More precisely,
\begin{equation} \label{eq234}\te{dim}_\ast(\mu) =\sup\set{s : \ul d_\mu(x) \geq s \te{ for } \mu\te{-a.e. } x}, \quad  \te{dim}^\ast(\mu) =\inf\set{s : \ul d_\mu(x) \leq s \te{ for } \mu\te{-a.e. } x},\end{equation}
and
\begin{equation} \label{eq235}\te{Dim}_\ast(\mu) =\sup\set{s : \ol d_\mu(x) \geq s \te{ for } \mu\te{-a.e. } x}, \quad  \te{Dim}^\ast(\mu) =\inf\set{s : \ol d_\mu(x) \leq s \te{ for } \mu\te{-a.e. } x}.\end{equation}
Hence, for $\mu$-a.e. $x\in \D R^k$, it follows that
\begin{equation*} 0\leq \te{dim}_\ast(\mu) \leq \ul d_\mu(x)\leq \te{dim}^\ast(\mu)\leq d, \te{ and } 0\leq \te{Dim}_\ast(\mu) \leq \ol d_\mu(x)\leq \te{Dim}^\ast(\mu)\leq d.\end{equation*}
If $\ul d_\mu(x)$ and $\ol d_\mu(x)$ are both constants for $\mu$-a.e. $x$, then we say that $\mu$ is `exact-dimensional' or `unidimensional'. We say that a measure $\mu$ has exact lower dimension $s$ if $\ul d_\mu(x)=s$ for $\mu$-a.e. $x$, and exact upper dimension $s$ if $\ol d_\mu(x)=s$ for $\mu$-a.e. $x$. Thus, from \eqref{eq234} and \eqref{eq235}, it follows that $\mu$ has exact lower dimension $s$ if and only if
\begin{equation} \label{eq23411}
\te{dim}_{\te{H}}(\mu)=\te{dim}_\ast(\mu)=\te{dim}^\ast(\mu)=s,
\end{equation}
and $\mu$ has exact upper dimension $s$ if and only if
\begin{equation} \label{eq23422}
\te{dim}_{\te{P}}(\mu)=\te{Dim}_\ast(\mu)=\te{Dim}^\ast(\mu)=s.
\end{equation}

For more details about the relationships between the different dimensions of measures, one is referred to \cite{F, P1,T1, Y}, and the references therein.

Let $P=[p_{ij}]_{1\leq i, j\leq N}$ be an $N\times N$ irreducible
row stochastic matrix, and $\set{S_{i} : 1\leq i\leq N}$ be a system
of contractive hyperbolic maps defined on a nonempty compact metric
space $X\sci \D R^k$ such that $\ul s_{i} d(x, y) \leq d(S_{i}(x),
S_{i}(y)) \leq \ol s_{i} d(x, y)$ for all $x, y\in X$, where $0<\ul
s_{i} \leq \ol s_{i} <1$, $1\leq i\leq N$. Then,  the collection
$\set{X;  S_{i}, p_{ij}  : 1\leq i, j\leq N}$ is called a
\tit{hyperbolic recurrent iterated function system} (hyperbolic
recurrent IFS) (see \cite{BEH}). Since $P$ is irreducible it follows
that (see \cite{F1}) there is a unique probability vector $p=(p_1,
p_2, \cdots, p_N)$ with $p_i>0$ for all $1\leq i\leq N$ such that
\[\sum_{i=1}^N p_ip_{ij}=p_j.\]
 Then, there exist unique nonempty compact sets $E_1, E_2, \cdots, E_N$ satisfying
\begin{equation} \label{eq41} E_i=\UU_{\{j: p_{ji}>0\}} S_{i}(E_j)
 \end{equation}
 for all $1\leq i\leq N$. Let $E=\uu_{i=1}^NE_i$. Then, by Caratheodory's extension theorem, it can be proved that there exists a unique Borel probability measure $\mu$ on $\D R^k$ with support $E$ such that $\mu$ satisfies:
\[\mu=\sum_{i=1}^N\sum_{j=1}^N p_{ji} \mu_j\circ S_{i}^{-1},\]
where $\mu_j:=\mu|_{E_j}$, i.e., $\mu_j$ is the restriction of $\mu$ on $E_j$, i.e., for any Borel $B \sci \D R^d$, we have $\mu_j(B) =\mu(B\ii E_j)$ for all $1\leq j\leq N$, (for some details, please see \cite{BEH}).
We say that the hyperbolic recurrent IFS satisfies the \tit{open set condition} (OSC) if there exist bounded nonempty open sets $U_1, U_2, \cdots, U_N$ with the property that
\[\UU_{\{j: p_{ji}>0\}} S_{i} (U_j) \sci U_i \te { and } S_{i}(U_j) \II S_{i}(U_k) =\es \te{ for } j \neq k \te{ with } p_{ji} p_{ki} >0.\]
The hyperbolic recurrent IFS satisfies the \tit{strong separation condition} (SSC) if
\[d\big(S_{i}(E_k), S_{j}(E_\ell)\big)>0 \te{ for all } ki\neq \ell j \te{ with } p_{ki} p_{\ell j} >0, \te{ where } 1\leq i, j, k, \ell \leq N.\]
It is a well-known fact that an hyperbolic recurrent IFS satisfies the open set condition if it
satisfies the strong separation condition.
In this paper, under the open set condition in Theorem~\ref{th1}, we have proved that for $\mu$-a.e. $x\in \D R^k$,
\begin{equation} \label{eq236}
\frac{\sum_{i=1}^N \sum_{j=1}^N p_i p_{ij} \log
p_{ij}}{\sum_{i=1}^Np_i  \log \ul s_{i}}\leq\ul d_\mu(x)\leq \ol
d_\mu(x) \leq\frac{\sum_{i=1}^N \sum_{j=1}^N p_i p_{ij} \log
p_{ij}}{\sum_{i=1}^Np_i  \log \ol s_{i}}.
\end{equation}
Thus, by \eqref{eq234},  \eqref{eq235}, and \eqref{eq236}, we have
\begin{equation} \label{eq89}  \frac{\sum_{i=1}^N \sum_{j=1}^N p_i p_{ij} \log
p_{ij}}{\sum_{i=1}^Np_i  \log \ul
s_{i}}\leq\te{dim}_\ast(\mu)\leq\ul d_\mu(x)\leq \ol
d_\mu(x) \leq\te{Dim}^\ast(\mu)\leq
\frac{\sum_{i=1}^N \sum_{j=1}^N p_i p_{ij} \log
p_{ij}}{\sum_{i=1}^Np_i  \log \ol s_{i}}.\end{equation}

The following theorem is known.

\begin{theorem} (see \cite[Theorem~2.1]{Z2}) \label{the90}
Let $\mu$ be a compactly supported probability measure on $\D R^k$.
Assume that there exist constants $C>0$ and $\gh>0$ such that
$\mu(B(x, \ep))\leq C \ep^\gh\te{ for every } x \in \D R^k \te{ and
all } \ep>0$. Then $\te{dim}_\ast(\mu) \leq \ul D(\mu)\leq \ol
D(\mu)\leq \te{Dim}^\ast(\mu).$
\end{theorem}
In Lemma~\ref{lemma00}, we have proved that the condition $\mu(B(x, \ep))\leq C \ep^\gh\te{ for every } x \in \D R^k$ given in Theorem~\ref{the90} also holds for hyperbolic recurrent IFS considered in this paper. Thus, by \eqref{eq89} and Theorem~\ref{the90}, we see that
\begin{equation}\label{eq237} \frac{\sum_{i=1}^N \sum_{j=1}^N p_i p_{ij} \log p_{ij}}{\sum_{i=1}^Np_i  \log \ul s_{i}}\leq \ul D(\mu) \leq \ol D(\mu)\leq \frac{\sum_{i=1}^N \sum_{j=1}^Np_i p_{ij} \log p_{ij}}{\sum_{i=1}^N  p_i \log \ol s_{i}}.\end{equation}
Under the strong separation condition in Theorem~\ref{th2}, we give
an independent proof of \eqref{eq237}.
Notice that if we assume that  $\ul s_{i}=\ol s_{i}=s_i$ and
$p_{ij}=p_j$ for all $1\leq i, j\leq N$, then \eqref{eq236} reduces
to
\[\ul d_\mu(x)=\ol d_\mu(x)=\frac{\sum_{i=1}^N p_i\log p_i}{\sum_{i=1}^N p_i\log s_i},\]
which is the result of Geronimo and Hardin in \cite{GH}, and \eqref{eq237} reduces to
\[\ul D(\mu) = \ol D(\mu)=\frac{\sum_{i=1}^N p_i\log p_i}{\sum_{i=1}^N p_i\log s_i},\]
which is the result of Graf-Luschgy in \cite{GL2}. In addition, Theorem~\ref{th1} generalizes a similar result of Deliu et al. in \cite{DGSH}, and Theorem~\ref{th2} generalizes a similar result of Roychowdhury et al. in \cite{RS} for recurrent self-similar measures. Thus, the results in this paper establish the connections with various dimensions of the measure $\mu$, and generalize many known results about local dimensions and quantization dimensions of measures.

\section{Basic definitions and results}
Let $X$ be a nonempty compact subset of the metric
space $\D R^k$ equipped with the metric $d$ inherited from $\D R^k$, such that $X=\te{cl(int}X)$. Let $N\geq 2$, and  $P=[p_{ij}]_{1\leq i, j\leq N}$ be an $N\times N$ irreducible row stochastic matrix, in other words, $p_{ij}\geq0$ for all $1\leq i,j\leq N$, $\sum_{j=1}^{N} p_{ij}=1$ for all $i$, and for each pair $1\leq i,j\leq N$ there exists a finite sequence of indicators
$i_1=i, i_2, \ldots, i_n=j$ with $i_1=i$ and $ i_n=j$ such that
\[p_{i_1i_2} p_{i_2i_3 } \cdots p_{i_{n-1}i_{n}}>0.\]
Let $\set{S_{i} : 1\leq i\leq N}$ be a system of contractive hyperbolic maps on $X$ as defined in the previous section.
Define $\gO$ and $T : \gO \to \gO$ by
\[\gO=\set{\go=(\go_i)_{i=1}^\infty : 1\leq \go_i\leq N, \, p_{\go_{i+1}\go_{i}}>0 \te{ for } i=1, 2, \cdots}\]
and
\[T : \gO  \ni (\go_1, \go_2, \go_3, \cdots) \mapsto (\go_2, \go_3, \cdots) \in \gO.\]
Let $\tilde d : \gO \times \gO \to \D R$ be defined by
\[\tilde d(\go, \gt) =2^{-n} \te{ if and only if } n=\min\set{m : \go_m\neq \gt_m}\]
for $\go=(\go_1, \go_2, \cdots), \, \gt=(\gt_1, \gt_2,  \cdots) \in \gO$. Then,  $\tilde d$ is a metric on $\gO$. With this metric $\gO$ becomes a compact metric space, and $T$ is called the left shift map on $\gO$.

For $n\geq 2$, let $W_n$ denote the set of all $n$-tuples $(\go_1,
\go_2,  \cdots, \go_{n})$ (called words of length $n$), which are
admissible with respect to $\gO$, i.e., there exists a sequence
$(\gt_1, \gt_2,  \cdots) \in \gO$ such that $\gt_1=\go_1, \gt_2=\go_2,
\cdots, \gt_n=\go_n$. Set $W=\UU_{n\geq 2} W_n$. If $\go=(\go_1, \go_2,
\cdots, \go_n)\in W$, then the set $\set{\gt \in \gO : \gt_i=\go_i,
1\leq i\leq n}$ is called a cylinder in $\gO$ of length $n$
generated by the word $\go$, and is denoted by $[\go]$. A cylinder of length zero is called the
empty cylinder. The set of all sequences in $\gO$ starting with the
symbol $i$ forms a cylinder of length 1, and is denoted by $[i]$. For $\go=(\go_1, \go_2, \cdots )
\in W \uu \gO$, if $n$ does not exceed the length of $\go,$ by
$\go|_n$ we mean $\go|_n=(\go_1, \go_2, \cdots, \go_n)$ and
$\go|_0=\es$, $\go$ is called an extension of $\gt \in W$ if
$\go|_{|\gt|}=\gt$, where $|\gt|$ represents the length of $\gt$.
For $\go =(\go_1, \go_2, \cdots, \go_n)$ and $\gt=(\gt_1, \gt_2,
\cdots, \gt_p)$ in $W,$ if $p_{\gt_1\go_n}>0$ by $\go \gt$ we mean
$\go\gt=(\go_1, \go_2, \cdots, \go_n, \gt_1, \cdots, \gt_p)$. For
$\go=(\go_1, \go_2, \cdots,  \go_n) \in W$, $n\geq 2$, let us write
\begin{align*}
S_\go &=S_{\go_1}\circ S_{\go_2}\circ \cdots \circ S_{\go_{n-1}}, \
p_\go  =p_{\go_n} p_{\go_{n}\go_{n-1}}\cdots  p_{\go_2\go_{1}},\\
P_\go&=p_{\go_{n}\go_{n-1}}\cdots  p_{\go_2\go_{1}}, \te{ and }
E_\go=S_\go(E_{\go_n}).
\end{align*}
From \eqref{eq41} it follows that the limit set of the recurrent hyperbolic IFS satisfies the following invariance equality (see \cite{BEH}):
\begin{equation} \label{eq511} E=\UU_{\go\in W_n}S_\go (E_{\go_n}) \te{ for } n\geq 2.\end{equation}
Let $\F B$ be the Borel sigma-algebra generated by the cylinders in $\gO$. For $\go=(\go_1, \go_2, \cdots,  \go_n) \in W$, define
$\gn[\go]=p_\go.$
Then, by the Caratheodory's extension theorem, $\gn$ can be extended to a unique Borel probability measure, which is also identified as $\gn$, on $(\gO, \F B)$. Clearly supp($\gn)=\gO$ and $T$ is an ergodic transformation on $(\gO, \F B, \gn)$. The transformation $T$ is called the \tit{Markov shift} with respect to the transition matrix $P$ and stationary distribution $p$, and $\gn$ is called the ergodic Markov measure on $(\gO, \F B)$. Since given $\go=(\go_1, \go_2,  \cdots ) \in \gO$ the diameters of
the compact sets $S_{\go|_n}(E_{w_n})$ converge to zero and since
they form a descending family, the set
\[\II_{n=2}^\infty S_{\go|_n} (E_{\go_n})\]
is a singleton and therefore, denoting its element by $\pi(\go)$, defines a coding map $\pi : \gO \to E$. Notice that the coding map $\pi$ is a continuous surjection.
Let $\mu$ be the image measure of the probability measure $\gn$ under the coding map $\pi$ on the limit set $E$, i.e., $\mu=\gn\circ \pi^{-1}$. Define
\[\pi_i:=\pi|_{[i]} \te{ and } \mu_i:=\mu|_{E_i},\]
i.e., $\pi_i$ is the restriction of $\pi$ on $[i]$, and $\mu_i$ is the restriction of $\mu$ on $E_i$. Thus, for $1\leq i\leq N$, we have $\mu_i=\gn \circ \pi_i^{-1}$, and under the strong separation condition, as defined in the previous section, we have
\[\mu_i=\sum_{j=1}^N p_{ji} \mu_j\circ S_{i}^{-1}, \te{ and }  \mu=\sum_{i=1}^N\mu_i.\]
In fact, for any $1\leq i\leq N$,
\[\mu_i(E_i)=\sum_{j=1}^N p_{ji} \mu_j\circ S_{i}^{-1}(E_i)=\sum_{j=1}^N p_{ji} \mu_j(E_j)=\sum_{j=1}^N  p_{ji}\Big(\gn\circ \pi_j^{-1}\Big)(E_j)=\sum_{j=1}^N  p_{ji}\gn[j],\]
which implies $\mu_i(E_i)=\sum_{j=1}^N  p_jp_{ji}=p_i$, and
\[\mu(E)=\sum_{i=1}^N \mu_i\Big(\bigcup_{k=1}^N E_k\Big)=\sum_{i=1}^N \mu_i(E_i)=\sum_{i=1}^N p_i=1.\]

 We call $\gG \sci W$ a \tit{finite maximal antichain} if $\gG$ is a
finite set of words in $W$, such that every sequence in $\gO$ is an
extension of some word in $\gG$, but no word of $\gG$ is an
extension of another word in $\gG$. Notice that as all words of $W$
are of length at least two, for any $\go\in \gG$, we have $|\go|\geq
2$.

We now prove the following lemma assuming the open set condition.
\begin{lemma} \label{lemma00}
Let $\mu$ be the Borel Probability measure on $\D R^k$ supported by the compact set generated by the hyperbolic recurrent IFS satisfying the open set condition. Then, there exist constants $C>0$ and $\gh>0$ such that
\begin{equation} \label{eq1} \mu(B(x, \ep))\leq C \ep^\gh\te{ for every } x \in \D R^k \te{ and all } \ep>0.\end{equation}
As a consequence, $e_n(\mu)<e_{n-1}(\mu)$ $(e_0(\mu):=\infty)$ and $\C C_n(\mu)\neq \es$ for every $n\in \D N$.
\end{lemma}

\begin{proof}
Let  $\epsilon_0 \in (0, 1]$ be arbitrary. By \cite[Lemma 12.3]{GL1}, it suffices to show \eqref{eq1} for every $x \in E$ and all $\ep \in (0, \ep_0)$. Write $\underline s_{\min}=\min\set{\ul s_{i} : 1\leq i\leq N}$, and for $\go=(\go_1, \go_2, \cdots, \go_n)\in W$, let
\[\go^-=(\go_1, \go_2, \cdots, \go_{n-1}), \, X_\go:=S_\go(X), \te{ and } \ul {s} _\go:=\ul s_{\go_1}\ul s_{\go_2}\cdots \ul s_{\go_{n-1}}.\]
 Without any loss of generality we assume that the diameter of $X$ is one. Let
\[\gG_\ep=\set{\go \in W: \ul s_{\go^{-}}\geq \ep> \ul s_\go} \te{ and } \gG_\ep(x)=\set{\go \in \gG_\ep : X_\go \ii B(x, \ep) \neq \es}, \]
for $x\in X$.
Notice that $\gG_\ep$ forms a finite maximal antichain.
We claim that there exists a constant $C>0$, independent of $x$ and $\ep$, such that $|\gG_\ep(x)|<C$.

Since $X$ has nonempty interior, $X$ contains a ball of radius $a$, where $a>0$ is a constant, and so for each $\go \in \gG_\ep$, the set $X_\go$ contains a ball of radius $a \ul s_\go\geq a \ul s_{\go^-}\ul s_{\min} \geq a\ul s_{\min}\ep$, and due to the open set condition, all such balls are disjoint. On the other hand, all $X_\go$ for $\go \in \gG_\ep$ are contained in the balls $B(x, 2\ep)$ for $x\in X$. Hence, comparing the volumes, for all $x\in X$, we have
\[|\gG_\ep(x)| (as_{\min}\ep)^k\leq (2\ep)^k\te{ which implies } |\gG_\ep(x)| \leq 2^k(as_{\min})^{-k},\]
where $k$ is dimension of the underlying space. Write $C=2^k(as_{\min})^{-k}$. Then,  $C>0$, and is independent of $x$ and $\ep$, and thus the claim is proved.
Now to prove the lemma, write
$$
P_{\max}=\max\set {\max\big\{p_1, p_2, \cdots, p_N\big\},
\max\big\{p_{ij} : 1\leq i, j\leq N\big\}}.
$$
We can take $B(x, \ep)$ such that $B(x, \ep) = \uu_{\go \in
\gG_\ep(x)} (X_\go \cap B(x, \ep) )$ for any $x\in X$. Then, we obtain
\[\mu(B(x, \ep))\leq \sum_{\go \in \gG_\ep(x)} \mu(X_\go)=\sum_{\go \in \gG_\ep(x)} \mu(E_\go)\leq C \max_{\go \in \gG_\ep(x)} p_\go\leq C \max_{\go \in \gG_\ep(x)} (P_{\max})^{|\go|}.\]
For $\go\in \gG_\ep(x)$, we have
$s_{\min}^{|\go|}\leq s_\go<\ep$, and so, $|\go|\geq \frac{\log \ep}{\log s_{\min}}$. Combining these facts, we have
\[\mu(B(x, \ep))\leq C(P_{\max})^{\frac{\log \ep}{\log s_{\min}}}=C\ep^{\log (P_{\max})/\log s_{\min}}.\]
The lemma follows by setting $\gh=\log (P_{\max})/\log s_{\min}$.
As in \cite[Proposition 3.1]{GL2}, one can see that the condition in \cite[Theorem 2.5]{GL2} is satisfied. As a consequence, $e_n(\mu)<e_{n-1}(\mu)$ $(e_0(\mu):=\infty)$ and $\C C_n(\mu)\neq \es$ for every $n\in \D N$.
\end{proof}

In the next section we state and prove the main results of the
paper.

\section{Main Results}
Theorem~\ref{th1} and Theorem~\ref{th2} contain the main results of the paper.
First, we state and prove the following theorem about the lower and the upper local dimensions of the measure $\mu$.

\begin{theorem}\label{th1} Let $\mu$ be the probability measure generated by the hyperbolic recurrent IFS $\set{X; S_{i}, p_{ij} :   1\leq i, j\leq N}$ satisfying the open set condition. Then,  for $\mu$-a.e. $x\in X$,
\[ \frac{\sum_{i=1}^N \sum_{j=1}^N p_i p_{ij} \log p_{ij}}{\sum_{i=1}^N p_i \log \ul s_{i}}\leq \ul d_\mu(x)\leq \ol d_\mu(x) \leq \frac{\sum_{i=1}^N \sum_{j=1}^N p_i p_{ij} \log p_{ij}}{\sum_{i=1}^N   p_i\log \ol s_{i}},\]
where $(p_1, p_2, \cdots, p_N)$ is the stationary distribution
associated with $[p_{ij}]_{1\leq i, j\leq N}$.
\end{theorem}

\begin{proof} Recall that for each hyperbolic map $S_i$, we have  $\ul s_{i} d(x, y) \leq d(S_{i}(x),
S_{i}(y)) \leq \ol s_{i} d(x, y)$ for all $x, y\in X$, where $0<\ul
s_{i} \leq \ol s_{i} <1$, $1\leq i\leq N$. To prove the theorem, in the first sight we assume that each $S_{i}$ is a similarity mapping with similarity ratio $s_{i}$ for $1\leq i\leq N$, where $s_i$ is a fixed number such that $\ul s_{i}\leq s_i\leq \ol s_{i}$.
Let $E_1, E_2, \cdots, E_N$ denote the components of the limit set $E$ such that
\[E=\UU_{i=1}^N E_i.\]
Fix $x\in E$, and let $\go=(\go_1, \go_2, \cdots) \in \gO$ be a code of $x$. Consider the ball $B(x, \rho)$ of diameter $\rho$ centered at $x$. Let $\ell$ be the least positive integer such that
\[S_{\go_1} \circ S_{\go_2} \circ \cdots \circ S_{\go_{\ell-1}} (E_{\go_\ell}) \sci B(x, \rho),\]
which implies
\[\mu(B(x, \rho))\geq p_{\go_\ell}  p_{\go_{\ell}\go_{\ell-1}}  \cdots  p_{\go_{2}\go_{1}}.\]
Set
\begin{align*} s_{\min}& :=\min\set{s_{i} : 1\leq i\leq N}, \   p_{\min} : =\min\set {p_1, p_2, \cdots, p_N}, \\
p_{\max} & :=\max\set {p_1, p_2, \cdots, p_N},  \ L_{\min} :=\min\set {\te{diam}(E_j) : 1\leq j\leq N}, \\
L_{\max}&:=\max\set {\te{diam}(E_j) : 1\leq j\leq N}.
 \end{align*} Then,  by the definition of $\ell$, we have
\[\prod_{i=1}^{\ell-1} s_{\go_{i}} \te{diam}(E_{\go_\ell})\leq 2\rho, \te{ which yields }  \prod_{i=1}^{\ell-1} s_{\go_i} \leq 2\rho L_{\min}^{-1},\]
and \[\prod_{i=1}^{\ell-2} s_{\go_i}\te{diam}(E_{\go_{\ell-1}})\geq \rho, \te{ which yields }
 \prod_{i=1}^{\ell-1} s_{\go_i}\geq \rho s_{\min}L_{\max}^{-1}.\]
Thus, by the definition of $\ell$, we have
\begin{equation} \label{eq34}
\rho s_{\min}L_{\max}^{-1}\leq \prod_{i=1}^{\ell-1} s_{\go_i}\leq 2\rho L_{\min}^{-1}.
\end{equation}
Therefore, by \eqref{eq34}, we have
\begin{align} \label{eq35} \mu(B(x, \rho)) & \geq\frac{(\prod_{i=1}^{\ell-1}p_{\go_{i+1}\go_i}) p_{\go_{\ell}}}{\prod_{i=1}^{\ell-1} s_{\go_i}}
\Big(\frac{\rho s_{\min}}{L_{\max}}\Big) \geq C_1\rho \prod_{i=1}^{\ell-1} \frac{ p_{\go_{i+1}\go_i}}{s_{\go_i}},  \end{align}
where $C_1=p_{\min} (s_{\min}/L_{\max})$. We next obtain an upper bound for $\mu(B(x, \rho))$.  For
$\rho>0$ and $(j_1, j_2, \cdots) \in \gO$, let $q$ be the least
integer such that
\begin{equation} \label{eq2} \prod_{m=1}^{q-1} s_{j_m}<\rho.\end{equation}
Let $S(\rho)$ be the set of such finite codes $(j_1, j_2, \cdots, j_q)$. By identifying $(j_1, j_2, \cdots, j_q)$ with the corresponding cylinder set in $\gO$, we see that $S(\rho)$ generates a partition of $\gO$. Let $U_1, U_2, \cdots, U_N$ be the open sets that arise in the open set condition. For each $(j_1, \cdots, j_q) \in S(\rho)$, set
\[U_{j_1, \cdots, j_q} =S_{j_1} \circ S_{j_2} \circ \cdots \circ S_{j_{q-1}}(U_{j_q}).\]
For a fixed $j_1 \in \{1, 2, \cdots, N\}$ the sets $\C
C_{j_1}(\rho)=\set{U_{j_1,j_2,\cdots,j_q} : (j_1, j_2, \cdots, j_q)
\in S(\rho)}$ are open disjoint sets. Recall that by hypothesis each $S_{i}$ is a similarity mapping, and so by \eqref{eq2}, there
are constants $c_1>0$ and $c_2>0$ independent of $\rho$ and $j_1$
such that each $U_{j_1, \cdots, j_q}$ contains a ball of radius
$c_1\rho$ and is contained in a ball of radius $c_2\rho$.
Consequently, Lemma 5.3.1 of Hutchinson (see \cite{H}) implies that
there are at most $\big((1+2c_2)/c_1\big)^k$ elements of $\C
C_{j_1}(\rho)$ whose closure meets $B(x, \rho)$, where $k$ is
the dimension of the underlying space. Let $E_{j_1,j_2,\cdots,
j_{q}}=S_{j_1} \circ S_{j_2} \circ \cdots \circ
S_{j_{q-1}}(E_{j_q})$ and let $I_\go(\rho)$ denote the set of $(j_1,
\cdots, j_q) \in S(\rho)$ such that $E_{j_1,j_2,\cdots, j_{q}}$
meets $B(x, \rho)$. The open set condition implies that
$E_{j_1,j_2,\cdots, j_{q}}\sci \ol U_{j_1,j_2,\cdots, j_q}$ and thus
$I_\go (\rho)$ contains at most $N\big((1+2c_2)/c_1\big)^k$
elements. Write $\tilde N=N((1+2c_2)/c_1)^k$. Then, we have
\begin{align*}
&\mu(B(x, \rho)) =\gn(\pi^{-1} (B(x, \rho))) \leq \sum_{(j_1, \cdots, j_q) \in I_\go(\rho)}  p_{j_q}p_{j_{q} j_{q-1}} \cdots p_{j_2j_1}\\
&=\sum_{(j_1, \cdots, j_q) \in I_\go(\rho)}\frac{ p_{j_q}p_{j_{q} j_{q-1}} \cdots p_{j_2j_1}}{s_{j_1} s_{j_{2}} \cdots s_{j_{q-1}}}   \Big(s_{j_1} s_{j_{2}} \cdots s_{j_{q-1}}\Big)\leq p_{\max}  \rho  \sum_{(j_1, \cdots, j_q) \in I_\go(\rho)}\Big (\prod_{i=1}^{q-1} \frac{p_{j_{i+1} j_{i}}}{s_{j_{i}}}\Big) \   [\te{ by } \eqref{eq2} ]\\
& \leq  p_{\max} \tilde N \rho \max_{(j_1, \cdots, j_q) \in I_\go(\rho)} \prod_{i=1}^{q-1}
\frac{p_{j_{i+1} j_{i}}}{s_{j_{i}}},
\end{align*}
which yields
\begin{equation} \label{eq3} \mu(B(x, \rho))\leq   C \rho \max_{(j_1, \cdots, j_q) \in I_\go(\rho)}\prod_{i=1}^{q-1}\frac{p_{j_{i+1} j_{i}}}{s_{j_{i}}},\end{equation}
where $C=p_{\max} \tilde N$. Now, let us define two functions $f, g : \gO \to \D R$ as follows:
\[f(\gt_1, \gt_2, \cdots) =\log p_{\gt_2\gt_1} \te{ and } g(\gt_1, \gt_2, \cdots) =\log s_{\gt_1},\]
for $\gt=(\gt_1, \gt_2, \cdots) \in \gO $. Then,  by Birkhoff's Ergodic Theorem, for $\gn$-a.e. $\gt\in\gO$, we have
\begin{align*}
\lim_{q\to \infty} \frac 1 q \sum_{i=0}^{q-1} f(T^i(\gt))& =\int_{\gO} f(\gt) d\gn \te{ and }
\lim_{q\to \infty} \frac 1 q \sum_{i=0}^{q-1} g(T^i(\gt))=\int_{\gO} g(\gt) d\gn.
\end{align*}
Notice that
\[ \frac 1 q \sum_{i=0}^{q-1} f(T^i(\gt))=\frac 1 q \sum_{i=0}^{q-1} f(\gt_{i+1}\gt_{i+2} \cdots)=\frac 1 q \sum_{i=0}^{q-1}\log p_{\gt_{i+2}\gt_{i+1}}=\frac 1 q \log \prod_{i=1}^q p_{\gt_{i+1}\gt_i}. \]
On the other hand, recalling the fact that $[i]=\set{(\gt_1, \gt_2, \cdots)\in \gO : \gt_1=i}$, we have
\begin{align*}
\int_\gO f(\gt) d\gn&=\sum_{i=1}^N \int_{[i]} f(\gt) d\gn=\sum_{i=1}^N \sum_{j=1}^N \int_{[ij]}f(\gt)d\gn=\sum_{i=1}^N \sum_{j=1}^N \gn[ij] \log p_{ji}\\
 &=\sum_{i=1}^N  \sum_{j=1}^N p_j p_{ji} \log p_{ji}=\sum_{i=1}^N \sum_{j=1}^N p_i p_{ij} \log p_{ij}.
\end{align*}
Hence, for $\gn$-a.e. $\gt\in \gO$, we have
\begin{equation} \label {eq13} \lim_{q\to \infty} \frac 1 q \log \prod_{i=1}^q p_{\gt_{i+1}\gt_i}=\sum_{i=1}^N \sum_{j=1}^N p_i p_{ij} \log p_{ij},\end{equation}
and similarly,
\begin{equation} \label {eq14}\lim_{q\to \infty} \frac 1 q \log \prod_{i=1}^q s_{\gt_i}=\sum_{i=1}^N \sum_{j=1}^N p_i p_{ij} \log s_i=\sum_{i=1}^N  p_i \log s_{i}.\end{equation}
By \eqref{eq2}, we have
\[ \prod_{m=1}^{q-1} s_{\gt_m}<\rho \leq \prod_{m=1}^{q-2} s_{\gt_m} \te{ implying } \rho s_{\min} \leq s_{\gt_1} s_{\gt_2}\cdots s_{\gt_{q-1}} =\prod_{i=1}^{q-1} s_{\gt_i}<\rho.\]
Thus, we see that $q\to \infty$ if $\rho\to 0$, and hence, for $\gn$-a.e. $\gt\in \gO$, we have
\begin{equation} \label{eq15} \lim_{\rho\to 0} \frac{\log \rho}{q}=\lim_{q\to \infty} \frac 1 q \log \prod_{i=1}^{q-1} s_{\gt_i} =\sum_{i=1}^N p_i\log s_{i}.
\end{equation}
Using \eqref{eq13}, \eqref{eq14} and \eqref{eq15}, for $\gn$-a.e. $\gt\in \gO$, we have
\begin{equation} \label{eq16}
\lim_{\rho\to 0} \log \Big (\prod_{i=1}^{q-1} \frac{p_{
\gt_{i+1}\gt_{i}}}{s_{\gt_{i}}}\Big)\Big /\log \rho= \frac{\sum_{i=1}^N \sum_{j=1}^N p_i p_{ij} \log p_{ij}}{\sum_{i=1}^N p_i  \log s_{i}}-1.
\end{equation}
By \eqref{eq3} and \eqref{eq16}, we have
\begin{equation} \label{eq17} \liminf_{\rho\to 0} \frac{\log\mu(B(x, \rho))}{\log \rho} \geq \frac{\sum_{i=1}^N \sum_{j=1}^N p_i p_{ij} \log p_{ij}}{\sum_{i=1}^N p_i  \log s_{i}}\geq \frac{\sum_{i=1}^N \sum_{j=1}^N p_i p_{ij} \log p_{ij}}{\sum_{i=1}^N p_i  \log \ul s_{i}}.\end{equation}
Equivalent to \eqref{eq15}, by \eqref{eq34}, for $\gn$-a.e. $\gt\in \gO$,  we have
\begin{equation} \label{eq18} \lim_{\rho\to 0} \frac{\log \rho}{\ell}=\lim_{\ell\to \infty} \frac 1 \ell \log \prod_{i=1}^{\ell-1} s_{\gt_i } =\sum_{i=1}^Np_i \log s_{i}.
\end{equation}
Thus, equivalent to \eqref{eq16}, the following relation is also true, for $\gn$-a.e. $\gt\in \gO$,
\begin{equation} \label{eq19}
\lim_{\rho\to 0} \log \Big (\prod_{i=1}^{\ell-1} \frac{p_{\gt_{i+1}
\gt_{i} }}{s_{\gt_{i} }}\Big)\Big /\log \rho=\frac{\sum_{i=1}^N \sum_{j=1}^N p_i p_{ij} \log p_{ij}}{\sum_{i=1}^N p_i  \log s_{i}}-1.
\end{equation}
By \eqref{eq35} and \eqref{eq19}, we have
\begin{equation} \label{eq20} \limsup_{\rho\to 0} \frac{\log\mu(B(x, \rho))}{\log \rho} \leq \frac{\sum_{i=1}^N \sum_{j=1}^N p_i p_{ij} \log p_{ij}}{\sum_{i=1}^N p_i  \log s_{i}}\leq \frac{\sum_{i=1}^N \sum_{j=1}^N p_i p_{ij} \log p_{ij}}{\sum_{i=1}^N p_i  \log \ol s_{i}}.\end{equation}
Thus, by \eqref{eq17} and \eqref{eq20}, for $\mu$-almost every $x\in X$,  we have
\[\frac{\sum_{i=1}^N \sum_{j=1}^N p_i p_{ij} \log p_{ij}}{\sum_{i=1}^N p_i \log \ul s_{i}}\leq \ul d_\mu(x)\leq  \ol d_\mu(x) =\frac{\sum_{i=1}^N \sum_{j=1}^N p_i p_{ij} \log p_{ij}}{\sum_{i=1}^Np_i  \log \ol s_{i}},\]
which completes the proof of the theorem.
\end{proof}

Now, we state and prove the following theorem, which gives
bounds for the lower and the upper quantization dimensions of the
measure $\mu$ with respect to the geometric mean error.
\begin{theorem}\label{th2} Let $\mu$ be the probability measure generated by the hyperbolic recurrent IFS $\set{X; S_{i}, p_{ij} :   1\leq i, j\leq N}$ satisfying the strong separation condition. Then,
\[\frac{\sum_{i=1}^N \sum_{j=1}^N p_i p_{ij} \log p_{ij}}{\sum_{i=1}^N  p_i  \log \ul s_{i}}\leq \ul D(\mu) \leq \ol D(\mu) \leq \frac{\sum_{i=1}^N \sum_{j=1}^N p_i p_{ij} \log p_{ij}}{\sum_{i=1}^N  p_i \log \ol s_{i}},\]
where $(p_1, p_2, \cdots, p_N)$ is the stationary distribution associated with $[p_{ij}]_{1\leq i, j\leq N}$.
\end{theorem}

To prove Theorem~\ref{th2} we need some lemmas and propositions. We
set
\[\ga_1:=\frac{\sum_{i=1}^N \sum_{j=1}^N p_i p_{ij} \log p_{ij}}{\sum_{i=1}^N  p_i  \log \ul s_{i}} \te{ and }\ga_2:=\frac{\sum_{i=1}^N \sum_{j=1}^N p_i p_{ij} \log p_{ij}}{\sum_{i=1}^N  p_i  \log \ol s_{i}}.\]
In the sequel for each $1\leq i\leq N$, let $\hat \mu_i$ be the conditional probability measure of $\mu$ given that $E_i$ has occurred, i.e., for any Borel $B\sci \D R^k$,
\[\hat \mu_i(B)=\frac{\mu(B\ii E_i)}{\mu(E_i)}=\frac 1 {p_i} \mu(B\ii E_i).\]
Notice that $\hat \mu_i$ has the support $E_i$ and $\mu_i=p_i\hat \mu_i$ for all $1\leq i\leq N$. Moreover,
\[\mu=\sum_{i=1}^N \mu_i=\sum_{i=1}^N \sum_{j=1}^N p_{ji} \mu_j\circ S_{i}^{-1} =\sum_{i=1}^N \sum_{j=1}^N p_jp_{ji} \hat \mu_j\circ S_{i}^{-1}.\]
Let us first prove the following lemma.

\begin{lemma} \label{lemma462}
For every $n\in \D N$ and $1\leq i\leq N$,
\[\hat e_n(\hat \mu_i)\leq\log \ol s_{i}  +\min \Big\{\frac {1}{p_i} \sum_{j=1}^N p_j p_{ji}  \hat e_{n_j}(\hat \mu_j) : n_j\geq 1, \sum_{j=1}^N n_j\leq n \Big\}.\]
\end{lemma}
\begin{proof} Let $n_j \in \D N$ with $\sum_{j=1}^N n_j\leq n$. Let $\ga_{j} \in \C C_{n_j}(\hat \mu_j)$. Since $\hat \mu_i=\frac 1 {p_i}\sum_{j=1}^N p_j p_{ji}\hat \mu_j\circ S_{i}^{-1}$, and for any $1\leq i\leq N$, $\te{card}(\uu_{j=1}^N S_{i}(\ga_{j})) \leq \sum_{j=1}^N n_j\leq n$, we have
\begin{align*}
&\hat e_n(\hat \mu_i) \\
&\leq \int \log  d(x, \UU_{k=1}^N S_{i}(\ga_{k})) d\hat \mu_i =\frac {1}{p_i} \sum_{j=1}^N p_j p_{ji} \int\log  d(x, \UU_{k=1}^N S_{i}(\ga_{k})) d(\hat \mu_j\circ S_i^{-1})\\
& \leq \frac {1}{p_i} \sum_{j=1}^N p_j p_{ji} \int\log  d(S_i(x), S_i(\ga_{j})) d\hat \mu_j\leq\frac {1}{p_i} \sum_{j=1}^N p_j p_{ji}  \log \ol s_{i} +\frac {1}{p_i} \sum_{j=1}^N p_j p_{ji}   \int\log  d(x, \ga_{j}) d\hat \mu_j\\
&=\log \ol s_{i} +\frac {1}{p_i} \sum_{j=1}^N p_j p_{ji}  \hat e_{n_j}(\hat \mu_j),
\end{align*}
and thus, the lemma follows.
\end{proof}

\begin{lemma} \label{lemma46}
For every $n\in \D N$,
\[\hat e_n(\mu)\leq \sum_{i=1}^N p_i\log \ol s_i +\min \Big\{ \sum_{i, j=1}^N p_j p_{ji}  \hat e_{n_{i}}(\hat \mu_j) : n_{i}\geq 1, \sum_{i=1}^N n_{i}\leq n \Big\}.\]
\end{lemma}
\begin{proof} Let $n_{i} \in \D N$ with $\sum_{i=1}^N n_{i}\leq n$. Let $\ga_{ij} \in \C C_{n_{i}}(\hat \mu_j)$. Recall that $\mu=\sum_{i,j=1}^N p_j p_{ji}\hat \mu_j\circ S_i^{-1}$, and notice that $\te{card}(\uu_{i=1}^N S_i(\ga_{ij})) \leq \sum_{i=1}^N n_{i}\leq n$. Thus, the rest of the proof follows in the similar way as the proof of Lemma~\ref{lemma462}.
\end{proof}

\begin{lemma} \label{lemma10}
Let $\gG\sci W$ be a finite maximal antichain. Let $C>\ga_2$ be arbitrary. Then, for all $n\geq |\gG|$,
\begin{equation*} \label{eq22} \hat e_n (\mu) \leq \frac 1 C\sum_{\gs \in \gG} p_\gs\log p_\gs +\min\left\{\sum_{\gs \in\gG}p_\gs \hat e_{n_\gs}(\hat \mu_{\gs_{|\gs|}}) : n_\gs \geq 1, \, \sum_{\gs\in \gG} n_\gs \leq n\right \}.\end{equation*}
\end{lemma}
\begin{proof}
Write $\ell(\gG)=\max \set{|\gs| : \gs \in \gG}$. We will prove the lemma by induction on $\ell(\gG)$. If $\ell(\gG)=0$, then $\ell(\gG)=\set{\es}$, and so, the lemma is obviously true. If $\ell(\gG)=1$, then the lemma is true by Lemma~\ref{lemma46}. Next, let $\ell(\gG)=k+1$, and assume that the lemma has been proved for all finite maximal antichains $\gG'$ with $\ell(\gG')\leq k$ for some $k\geq 1$. Define
\begin{align*}
\gG_1&=\set{\gs \in \gG : |\gs| <\ell(\gG)}, \\
\gG_2&=\set{\gs^-  : \gs \in \gG \te{ and } |\gs| =\ell(\gG)},
\end{align*}
and \[\gG_0=\gG_1\UU \gG_2.\]
It is easy to see that $\gG_0$ is a finite maximal antichain with $\ell(\gG_0) \leq k$.  Let $\gs\ast j$ denote the word $\gs j$. Then, for $n\geq |\gG|$ and $(n_\gs)_{\gs \in \gG}$ with $n_\gs \geq 1$ and $\sum_{\gs \in \gG} n_\gs \leq n$,  we have
\begin{align*}
&a:=\frac 1 C \sum_{\gs \in \gG} p_\gs\log p_{\gs} +\sum_{\gs \in\gG}p_\gs \hat e_{n_\gs}(\hat \mu_{\gs_{|\gs|}})\\
&=\frac 1 C\Big [\sum_{\gs \in \gG_1} p_\gs\log p_{\gs}+\sum_{\gs
\in \gG_2}\sum_{j=1}^N  p_jp_{j\gs_{|\gs|}} p_{\gs_{|\gs|}
\gs_{|\gs|-1}}   \cdots p_{\gs_{2} \gs_1}
\Big(\log(p_jp_{j\gs_{|\gs|}})
 + \log(p_{\gs_{|\gs|} \gs_{|\gs|-1}}   \cdots p_{\gs_{2} \gs_1})\Big)\Big]\\&+\sum_{\gs \in\gG_1}p_\gs \hat e_{n_\gs}(\hat \mu_{\gs_{|\gs|}})+\sum_{\gs \in\gG_2}\sum_{j=1}^N   p_jp_{j\gs_{|\gs|}} p_{\gs_{|\gs|} \gs_{|\gs|-1}}   \cdots p_{\gs_{2} \gs_1}  \hat e_{n_{\gs\ast j}}(\hat \mu_j)\\
&=\frac 1 C\Big[\sum_{\gs \in \gG_1} p_\gs\log p_{\gs}+\sum_{\gs
\in \gG_2} p_{\gs_{|\gs|} \gs_{|\gs|-1}}   \cdots p_{\gs_{2} \gs_1}
\Big(\sum_{j=1}^N  p_jp_{j\gs_{|\gs|}}\log(p_jp_{j\gs_{|\gs|}})\Big)
+ \sum_{\gs \in\gG_2} p_\gs \log P_\gs \Big]\\&+\sum_{\gs
\in\gG_1}p_\gs \hat e_{n_\gs}(\hat \mu_{\gs_{|\gs|}})+\sum_{\gs
\in\gG_2} p_{\gs_{|\gs|} \gs_{|\gs|-1}}   \cdots p_{\gs_{2} \gs_1}
\Big(\sum_{j=1}^N   p_jp_{j\gs_{|\gs|}}\hat e_{n_{\gs\ast j}}(\hat
\mu_j)\Big).
\end{align*}
If we set $n_\gs=\sum_{j=1}^Nn_{\gs \ast j}$ for $\gs \in \gG_2$, by Lemma~\ref{lemma462}, we obtain
\begin{equation} \label{eq2341} \sum_{j=1}^N   p_jp_{j\gs_{|\gs|}}\hat e_{n_{\gs\ast j}}(\hat \mu_j) \geq  p_{\gs_{|\gs|}}\hat e_{n_\gs}(\hat \mu_{\gs_{|\gs|}})- p_{\gs_{|\gs|}}\log \ol s_{\gs_{|\gs|}},\end{equation}
which yields
\begin{align*} & \sum_{\gs \in\gG_2} p_{\gs_{|\gs|} \gs_{|\gs|-1}}   \cdots p_{\gs_{2} \gs_1}  \Big(\sum_{j=1}^N   p_jp_{j\gs_{|\gs|}}\hat e_{n_{\gs\ast j}}(\hat \mu_j)\Big)
\geq \sum_{\gs \in \gG_2} p_\gs\hat e_{n_\gs}(\hat \mu_{\gs_{|\gs|}})-\sum_{\gs \in \gG_2} p_\gs\log \ol s_{\gs_{|\gs|}}.\end{align*}
By the hypothesis,  $C>\frac{\sum_{i=1}^N \sum_{j=1}^N p_i p_{ij} \log p_{ij}}{\sum_{i=1}^N  p_i \log \ol s_{i}}\geq \frac{\sum_{i=1}^N \sum_{j=1}^N p_i p_{ij} \log (p_ip_{ij})}{\sum_{i=1}^N  p_i \log \ol s_{i}}$, which implies the fact that
 \[\sum_{i=1}^N\Big (\sum_{j=1}^N p_jp_{ji}\log(p_j p_{ji})
- C p_i\log \ol s_{i}\Big)>0.\] To prove the lemma, we assume that
\begin{equation} \label{eq2342}
\sum_{\gs
\in \gG_2} p_{\gs_{|\gs|} \gs_{|\gs|-1}}   \cdots p_{\gs_{2} \gs_1}
\Big(\sum_{j=1}^N  p_jp_{j\gs_{|\gs|}}\log(p_jp_{j\gs_{|\gs|}})- C
p_{\gs_{|\gs|}}\log \ol s_{\gs_{|\gs|}}\Big)\geq 0.
\end{equation}
 Again, for any $\gs \in \gG_2$, we have $P_\gs\geq p_\gs$. Thus, using \eqref{eq2341} and \eqref{eq2342}, we have
\begin{align*}
a &\geq \frac 1 C\Big [\sum_{\gs \in \gG_1} p_\gs\log p_\gs+C
\sum_{\gs \in \gG_2} p_\gs  \log \ol s_{\gs_{|\gs|}}+\sum_{\gs \in
\gG_2} p_\gs \log p_\gs\Big ]+\sum_{\gs \in\gG_1}p_\gs \hat
e_{n_\gs}(\hat \mu_{\gs_{|\gs|}})
\\
&\qquad \qquad+\sum_{\gs \in \gG_2} p_\gs\hat e_{n_\gs}(\hat \mu_{\gs_{|\gs|}})-\sum_{\gs \in \gG_2} p_\gs\log \ol s_{\gs_{|\gs|}}\\
 &=\frac 1 C\sum_{\gs \in \gG_0} p_\gs\log p_\gs+\sum_{\gs \in\gG_0}p_\gs \hat e_{n_\gs}(\hat \mu_{\gs_{|\gs|}}).
\end{align*}
Since
\[\sum_{\gs \in \gG_0}n_\gs =\sum_{\gs \in \gG_1}n_\gs+\sum_{\gs \in \gG_2}n_\gs\leq n,\]
by the induction hypothesis, we obtain
\[\hat e_n (\mu) \leq \frac 1 C\sum_{\gs \in \gG_0} p_\gs\log p_\gs +\sum_{\gs \in\gG_0}p_\gs \hat e_{n_\gs}(\hat \mu_{\gs_{|\gs|}})\leq a,\] which yields the lemma.

\end{proof}

The following lemma plays an important role in the paper.
\begin{lemma}
Let $\hat p_{\min}=\min\Big\{\min \big\{p_j : 1\leq j\leq N\big\}, \min
\big\{p_{ij} : p_{ij}>0, \, 1\leq i, j\leq N\big\}\Big\}$. For
$0<\ep<1$, write
\[\gG(\ep) =\set{\gs\in W : p_{\gs^-} \geq \ep>p_\gs},\]
where $\gs^-$ is the word obtained from $\gs$ by deleting the last letter of $\gs$. Let $m, n\in \D N$ with $m$ fixed and $\frac m n<\hat p_{\min}^2$. Write $\ep_n=\frac m n \hat p_{\min}^{-2}$. Then,
\begin{equation} \label{eq31} \hat e_n(\mu) \leq \frac 1 {C} \sum_{\gs \in \gG(\ep_n)}p_\gs \log p_\gs +\sum_{j=1}^N \hat e_m(\hat \mu_j).\end{equation}
\end{lemma}
\begin{proof}
We claim that $\gG(\ep_n)$ is a finite maximal antichain. To prove the claim we proceed as follows:
For $(\gs_1, \gs_2, \cdots) \in \gO$, let $k(\gs)$ be the least positive integer such that
\[p_{(\gs_1, \gs_2, \cdots, \gs_{k(\gs))}}<\ep_n\leq p_{(\gs_1, \gs_2, \cdots, \gs_{k(\gs)-1})}.\]
Then, $(\gs_1, \gs_2, \cdots, \gs_{k(\gs)}) \in \gG(\ep_n)$, and $(\gs_1, \gs_2, \cdots)$ is an extension of $(\gs_1, \gs_2, \cdots, \gs_{k(\gs)})$. Next, let  $(\go_1, \go_2, \cdots, \go_k)$, and $(\gt_1, \gt_2, \cdots, \gt_\ell)$ be two elements in $\gG(\ep_n)$. We show that they are not extensions of each other. For the sake of contradiction, assume that $(\gt_1, \gt_2, \cdots, \gt_\ell)$ is an extension of $(\go_1, \go_2, \cdots, \go_k)$. Then, $k<\ell$, and
\begin{equation*} \label{eq77} (\gt_1, \gt_2, \cdots, \gt_k)=(\go_1, \go_2, \cdots, \go_k).
\end{equation*}
Again, as $(\go_1, \go_2, \cdots, \go_k)\in \gG(\ep_n)$, we have
\[p_{(\go_1, \go_2, \cdots, \go_k)}<\ep_n\leq (\go_1, \go_2, \cdots, \go_{k-1}),\]
yielding \[p_{(\gt_1, \gt_2, \cdots, \gt_k)}<\ep_n\leq p_{(\gt_1, \gt_2, \cdots, \gt_{k-1})},\]
which contradicts the fact that $\ell$ is the least positive integer for which
 \[p_{(\gt_1, \gt_2, \cdots, \gt_\ell)}<\ep_n\leq p_{(\gt_1, \gt_2, \cdots, \gt_{\ell-1})}\]
 as $(\gt_1, \gt_2, \cdots, \gt_\ell)\in \gG(\ep_n)$. Similarly, we can show that if $(\go_1, \go_2, \cdots, \go_k)$ is an extension of $(\gt_1, \gt_2, \cdots, \gt_\ell)$, then a contradiction arises. Hence, any two words in $\gG(\ep_n)$ are not extensions of each other.
 Thus, we see that $\gG(\ep_n)$ is a finite maximal antichain, which is the claim.

 Notice that for any word $\gs=(\gs_1, \gs_2,\cdots, \gs_{|\gs|})$, we have
 \[\frac{p_\gs} {p_{\gs^{-}}}=\frac {p_{\gs_{|\gs|}}p_{\gs_{|\gs|}\gs_{|\gs|-1}}}{p_{\gs_{|\gs|-1}}}>p_{\gs_{|\gs|}}p_{\gs_{|\gs|}\gs_{|\gs|-1}}\geq \hat p_{\min}^{2}.\]
Hence, if $\gs \in \gG(\ep_n)$, we have
\[1=\sum_{\gs \in \gG(\ep_n)} p_\gs \geq \sum_{\gs \in \gG(\ep_n)} p_{\gs^{-}}  \hat p_{\min}^{2}\geq  \sum_{\gs \in \gG(\ep_n)} \ep_n \hat p_{\min}^2=\frac m n |\gG(\ep_n)|,  \]
which implies $n\geq m  |\gG(\ep_n)|$. Write $\gG(\ep_n, j)=\set{(\gs_1, \cdots, \gs_{|\gs|}) \in \gG(\ep_n) : \gs_{|\gs|}=j}$, and then $\gG(\ep_n)=\UU_{j=1}^N \gG(\ep_n, j)$. Choosing $n_\gs=m$ for every $\gs \in \gG(\ep_n)$ in Lemma~\ref{lemma10}, we have
\begin{align*} \hat e_n (\mu) & \leq \frac 1 {C}\sum_{\gs \in \gG(\ep_n)} p_\gs\log p_\gs +\sum_{\gs \in\gG(\ep_n)}p_\gs \hat e_{m}(\hat \mu_{\gs_{|\gs|}})=\frac 1 {C}\sum_{\gs \in \gG(\ep_n)} p_\gs\log p_\gs +\sum_{j=1}^N \sum_{\gs \in\gG(\ep_n, j)}p_\gs \hat e_{m}(\hat \mu_j)\\
&\leq \frac 1 {C}\sum_{\gs \in \gG(\ep_n)} p_\gs\log p_\gs +\Big (\sum_{j=1}^N \sum_{\gs \in\gG(\ep_n, j)}p_\gs\Big) \Big(\sum_{j=1}^N \hat e_{m}(\hat \mu_j)\Big)\\
&=\frac 1 {C} \sum_{\gs \in \gG(\ep_n)}p_\gs \log p_\gs +\sum_{j=1}^N \hat e_m(\hat \mu_j),
\end{align*}
and thus the lemma is proved.
\end{proof}

Let us now prove the following proposition.
\begin{prop} \label{prop111}
Let $C>\ga_2$ be arbitrary. Then,
\[\limsup_{n\to \infty} n^{1/C} e_n(\mu)\leq \hat p_{\min}^{-2/C}\inf_{m\geq 1} m^{1/C}\prod_{j=1}^N e_m(\hat \mu_j)<+\infty.\]
\end{prop}
\begin{proof}
It is enough to prove that
\begin{equation} \label{eq32} \limsup_{n\to \infty} (\log n+C\hat e_n(\mu)) \leq  -2\log \hat p_{\min} +\inf_{m\geq 1}\Big(\log m +C\sum_{j=1}^N \hat e_m(\hat \mu_j)\Big).\end{equation}
Given $m\in \D N$, by \eqref{eq31}, for $\ep_n=\frac mn \hat p_{\min}^{-2}$, we obtain
\[\log n+C \hat e_n(\mu)\leq \sum_{\gs \in \gG(\ep_n)}p_\gs \log p_\gs-\log \ep_n -2\log \hat p_{\min}+\log m +C \sum_{j=1}^N \hat e_m(\hat \mu_j) \]
for all but finitely many $n$. To prove \eqref{eq32}, it is therefore enough to prove that
\[\limsup_{n\to \infty} \Big[\sum_{\gs \in \gG(\ep_n)} p_\gs \log p_{\gs}-\log \ep_n\Big ]\leq 0.\]
Since $p_\gs <\ep_n$ for all $\gs \in \gG(\ep_n)$, we have $\log \ep_n \geq \log p_{\gs}$, and hence
\[\sum_{\gs \in \gG(\ep_n)} p_\gs \log p_{\gs}\leq \log \ep_n,\]
which proves \eqref{eq32}, and thus the proposition follows.
\end{proof}

In order to prove Proposition~\ref{prop112} we need the following
lemma.
\begin{lemma}\label{lemma461}
Let the hyperbolic recurrent IFS satisfy the strong separation condition. Then,  for $1\leq i\leq N$,
\[\hat e_n(\hat \mu_i)\geq  \log \ul s_{i} +\max \left\{\frac{1}{p_i} \sum_{j=1}^N p_j  p_{ji}\hat e_{n_j}(\hat \mu_j) : n_j\geq 1, \sum_{j=1}^N n_j\leq n \right\},\]
for all but finitely many $n\in \D N$.
\end{lemma}
\begin{proof}
Let
$\gd=\min\set {d(S_{i}(E_k), S_{j}(E_\ell)) : ki \neq \ell j \te{ with } 1\leq i, j, k, \ell \leq N}$ and let $\ga_n \in \C C_n(\hat \mu_i)$, $n\in \D N$. Then,  $\gd>0$. Now proceeding along the similar lines as \cite[Lemma~5.9]{GL2}, it can be proved that
\[\lim_{n\to \infty} \max_{x\in E_i}d(x, \ga_n)=0,\]
and so there exists a positive integer $n_0$ such that $\max_{x\in E_i} d(x, \ga_n)<\frac \gd 2$ for all $n\geq n_0$. For $1\leq k\leq N$, set $\ga_{n, k}=\set{ a \in \ga_n : W(a|\ga_n)\II S_{i} (E_k) \neq \es}$, where $W(a|\ga_n)$ is the Voronoi region generated by $a\in \ga_n$ (see \cite{GL2} for more details on Voronoi regions). Then,  $\ga_{n, k}\neq \es$ and $\ga_{n, k}\ii\ga_{n, \ell}=\es$ for $1\leq k\neq \ell\leq N$ and $n\geq n_0$. Using $\hat \mu_i=\frac {1}{p_i} \sum_{j=1}^N p_j p_{ji}\hat \mu_j\circ S_{i}^{-1}$, for all $n\geq n_0$,  we obtain
\begin{align*} \hat e_n(\hat \mu_i) &=\int \log d(x, \ga_n) d\hat \mu_i=\frac {1}{p_i} \sum_{j=1}^N p_j p_{ji} \int \log d(S_{i}(x), \ga_n) d\hat \mu_j\\
&=\frac {1}{p_i} \sum_{j=1}^N p_j p_{ji}\int \log d(S_{i}(x), \ga_{n, j}) d\hat \mu_j\\
&\geq \frac {1}{p_i} \sum_{j=1}^N p_j p_{ji} \log \ul s_{i}+\frac {1}{p_i} \sum_{j=1}^N p_j p_{ji} \int \log d(x, S_{i}^{-1}(\ga_{n, j})) d\hat \mu_j\\
&\geq \log \ul s_{i}+\frac {1}{p_i} \sum_{j=1}^N p_j p_{ji} \hat e_{n_j}(\hat \mu_j),
\end{align*}
where $n_j=\te{card}(\ga_{n, j})\geq 1$. Since $n=\te{card}(\ga_{n})=\sum_{j=1}^N n_j$, this proves the lemma.
\end{proof}
We now prove the following lemma.
\begin{lemma}   \label{lemma45}  Since $\mu=\sum_{i=1}^N p_i\hat \mu_i$, we have
\[\hat e_n(\mu) \geq \sum_{i=1}^N p_{i} \hat e_n (\hat \mu_i).\]
\end{lemma}
\begin{proof}
Let $\ga \in \C C_n(\mu)$. Then,
\begin{align*}
\hat e_n(\mu) &= \int \log  d(x,\ga) d\mu =\sum_{i=1}^N p_{i} \int\log  d(x, \ga) d\hat \mu_i\geq  \sum_{i=1}^N p_{i}\hat e_n(\hat \mu_i),
\end{align*}
and thus the lemma follows.
\end{proof}

The following lemma is the English version of \cite[page 23, Lemma 1]{T}.
\begin{lemma} \label{lemma1001}
If $(s_1,...,s_n)$ and $(y_1,...,y_n)$  are two positive numbers
vectors with $\sum_i y_i  \geq \sum_i  s_i$, then
$$
\sum_{i=1}^n  y_i \log(\frac{y_i}{s_i}) \geq 0,
$$
and the  equality is valid if and only if $y_i = s_i$ for all $i$.
\end{lemma}

\begin{proof}
We start with the elementary inequality:
$\log(x)\leq x-1\te{ for all } x>0,$
which implies the fact that
\[\sum_{i=1}^n y_i\log (\frac{s_i}{y_i})\leq \sum_{i=1}^n y_i(\frac{s_i}{y_i}-1) =\sum_{i=1}^n s_i-\sum_{i=1}^n y_i\leq 0.\]
Thus, the lemma follows.
\end{proof}
\begin{prop} \label{prop112}
Let the hyperbolic recurrent IFS satisfy the strong separation condition, and let $\ga_1$ be defined as before. Then,
\[\inf_{n\in \D N} n^{1/{\ga_1}} e_n(\mu)>0.\]
\end{prop}
\begin{proof} The proposition will be proved if we can prove that
\[\inf_{n\in \D N} (\log n + {\ga_1}\hat e_n(\mu))>-\infty.\]
By Lemma~\ref{lemma461}, there is an $n_0 \in \D N$, such that
\[\hat e_n(\hat \mu_i) \geq \log \ul s_{i}+\min\left \{\frac 1{p_i} \sum_{j=1}^Np_j p_{ji}\hat e_{n_j}(\hat \mu_j) : n_j\geq 1, \sum_{j=1}^N n_j\leq n\right\},\]
 for all $1\leq i\leq N$ and all $n\geq n_0$. Since $\hat e_n(\hat \mu_i)>-\infty$ for all $n\in \D N$, we have
 $$c=\min \Big\{\frac 1 {\ga_1} \log n+\hat e_n(\hat \mu_i) : n\leq n_0\Big\}>-\infty.$$ By induction, we now prove that
\[\hat e_n(\hat \mu_i) \geq c-\frac 1 {\ga_1} \log n,\]
for all $n\in \D N$. For $m\leq n_0$, this is true by the definition of $c$. Let $m>n_0$, and assume that the inequality holds for all $n<m$. Then,
\begin{align*}
\hat e_m(\hat \mu_i) &\geq \log \ul s_{i}+\min\left\{\frac 1{p_i} \sum_{j=1}^Np_j p_{ji}\hat e_{n_j}(\hat \mu_j) : n_j\geq 1, \sum_{j=1}^N n_j\leq m\right\}\\
&\geq \log \ul s_{i}+\min\left \{\frac 1{p_i} \sum_{j=1}^Np_j p_{ji} c-\frac 1 {\ga_1} \frac 1{p_i} \sum_{j=1}^Np_j p_{ji} \log n_j : n_j\geq 1, \sum_{j=1}^N n_j\leq m\right \}\\
&=\log \ul s_{i}+c-\frac 1 {\ga_1}\log m -\frac 1 {\ga_1} \max \left
\{ \frac 1{p_i} \sum_{j=1}^Np_j p_{ji} \log \frac {n_j}{m} : n_j\geq
1, \sum_{j=1}^N n_j\leq m\right\}.
\end{align*}
Using Lemma~\ref{lemma1001} and the fact that $\frac 1{p_i} \sum_{j=1}^Np_j p_{ji}=1$,
we obtain
\[\frac 1{p_i} \sum_{j=1}^Np_j p_{ji}\log \frac{n_j}{m} \leq \frac 1{p_i} \sum_{j=1}^Np_j p_{ji} \log \Big(\frac{p_j p_{ji}}{p_i}\Big),\]
for all $n_j\geq 1$ with $\sum_{j=1}^N n_j\leq m$. Thus we have,
\[\hat e_m(\hat \mu_i) \geq\log \ul s_{i}+c-\frac 1 {\ga_1}\log m -\frac 1 {\ga_1} \frac 1{p_i} \sum_{j=1}^Np_j p_{ji} \log \Big(\frac{p_j p_{ji}}{p_i}\Big),\]
which by Lemma~\ref{lemma45} yields
\begin{equation} \label{eq000}\hat e_m(\mu) \geq \sum_{i=1}^N p_i\hat e_m(\hat \mu_i) \geq   \sum_{i=1}^N p_i \log \ul s_{i}+c-\frac 1 {\ga_1}\log m-\frac 1 {\ga_1} \sum_{i, j=1}^Np_j p_{ji}\log \Big(\frac{p_j p_{ji}}{p_i}\Big).\end{equation}
Notice that
\begin{align*}
&\sum_{i, j=1}^N p_jp_{ji}\log \Big(\frac{p_j p_{ji}}{p_i}\Big)=\sum_{i, j=1}^N p_jp_{ji}\log p_{ji}+ \sum_{i, j=1}^N p_jp_{ji} \log p_j -\sum_{i, j=1}^N p_jp_{ji}\log p_i\\
&=\sum_{i, j=1}^N p_jp_{ji}\log p_{ji}+ \sum_{j=1}^N p_j\log p_j -\sum_{i=1}^N p_i\log p_i=\sum_{i, j=1}^N p_jp_{ji}\log p_{ji},
\end{align*}
and thus,
\begin{equation} \label{eq0001} \frac 1 \ga_1 \sum_{i, j=1}^N p_jp_{ji}\log \Big(\frac{p_j p_{ji}}{p_i}\Big)=\sum_{i=1}^Np_i\log \ul s_i.\end{equation}
Hence, by \eqref{eq000}  and \eqref{eq0001}, we have
\[\hat e_m(\mu) \geq c-\frac 1 {\ga_1}\log m.\]
This implies
\[\inf_{n\in \D N} (\log n+{\ga_1}\hat e_n(\mu))\geq c{\ga_1}>-\infty,\]
and hence the proposition is proved.
\end{proof}

\nd \tbf{Proof of Theorem~\ref{th2}.} Proposition~\ref{prop111} tells us that  $\limsup_{n\to \infty}  n^{1/C} e_n(\mu)<\infty$, which by Proposition~\ref{prop1} implies that $\ol D(\mu) \leq C$. Since $C>\ga_2$ is arbitrary, we have $\ol D(\mu) \leq \ga_2$.  Proposition~\ref{prop112} tells that $\liminf_{n\to \infty} n^{1/{\ga_1}} e_n(\mu)>0$, which by Proposition~\ref{prop1} implies $\ul D(\mu)\geq \ga_1$. Thus the proof of Theorem~\ref{th2} is complete.
\qed

\begin{remark}
Let $\ga_2$ be the upper bound for
the upper local dimension of the probability measure $\mu$ generated
by the hyperbolic recurrent IFS $\set{X; S_{i}, p_{ij} : 1\leq i,
j\leq N}.$  Then, the following problem remains open:
\[\te{Is } \limsup_{n\to \infty} n^{1/{\ga_2}} e_n(\mu)<+\infty ?\]
\end{remark}

\subsection*{Acknowledgement} The authors are grateful to the referees for their valuable comments and suggestions.

\end{document}